\documentclass[10pt]{article}
\usepackage[letterpaper,total={4.5in,7.125in}]{geometry}
\usepackage[T1]{fontenc}

\usepackage{amsmath}             
\usepackage{amsthm}
\usepackage[capitalize,nameinlink]{cleveref}
\usepackage{enumerate}
\usepackage{url}
\usepackage{tikz}

\newtheorem{theorem}{Theorem}
\newtheorem{corollary}[theorem]{Corollary}
\newtheorem{proposition}[theorem]{Proposition}
\newtheorem{lemma}[theorem]{Lemma}
\newtheorem{definition}[theorem]{Definition}
\newtheorem{example}[theorem]{Example}

\newcommand{\ff}{F(m\lambda;\lambda)}

\long\def\symbolfootnote[#1]#2{\begingroup%
\def\thefootnote{\fnsymbol{footnote}}\footnote[#1]{#2}\endgroup}

\allowdisplaybreaks[3]

\begin{document}

\title{A new representation of mutually orthogonal frequency squares}
\author{Jonathan Jedwab \and Tabriz Popatia}
\date{8 March 2020 (revised 14 November 2020)}
\maketitle

\begin{center}
\emph{To Gary MacGillivray, in appreciation of his mathematics research, teaching, and service over several decades}

\end{center}

\symbolfootnote[0]{
Department of Mathematics,
Simon Fraser University, 8888 University Drive, Burnaby BC V5A 1S6, Canada.
\par
J.~Jedwab is supported by an NSERC Discovery Grant. 
\par
T.~Popatia was supported by an NSERC Undergraduate Student Research Award.
\par
Email: {\tt jed@sfu.ca}, {\tt tabriz\_popatia@sfu.ca}
}

\begin{abstract}
Mutually orthogonal frequency squares (MOFS) of type $\ff$ generalize the structure of mutually orthogonal Latin squares: rather than each of $m$ symbols appearing exactly once in each row and in each column of each square, the repetition number is~$\lambda \ge 1$.
A classical upper bound for the number of such MOFS is $\frac{(m\lambda-1)^2}{m-1}$.
We introduce a new representation of MOFS of type $\ff$, as a linear combination of $\{0,1\}$ arrays. We use this representation to give an elementary proof of the classical upper bound, together with a structural constraint on a set of MOFS achieving the upper bound.
We then use this representation to establish a maximality criterion for a set of MOFS of type $\ff$ when $m$ is even and $\lambda$ is odd, which simplifies and extends a previous analysis \cite{britz} of the case when $m=2$ and $\lambda$ is odd.
\end{abstract}

\section{Introduction}
\label{sec:intro}
Latin squares are a fundamental concept in combinatorial design theory, whose study is at least 300 years old \cite[p.~12]{anderson-colbourn-dinitz-griggs}.
A frequency square is a generalization of a Latin square, introduced by MacMahon \cite{macmahon} in 1898 under the name ``quasi-latin square'', subsequently studied in the 1940s by Finney~\cite{finney1}, and named in 1969 by Hedayat \cite{hedayat-phd} (see \cite{laywine-mullen} for a survey).

\begin{definition}
A \emph{frequency square (F-square) of type}  $\ff$ is an $m\lambda \times m\lambda$ array with elements belonging to the symbol set $\{1,2, \dots ,m\}$, where each symbol $j$ appears exactly $\lambda$ times in each row and in each column.
\end{definition}
\noindent
We note that some authors write $F(m\lambda;\lambda^m)$ instead of $\ff$.
The special case of an F-square of type $F(m;1)$ is a Latin square of order $m$. An F-square of type $\ff$ corresponds to a statistical experimental design offering more flexibility than a design based on a Latin square \cite[Section~2]{hedayat-seiden}.

\begin{definition}
Two F-squares $S$ and $S'$ of type $F(m\lambda;\lambda)$ are \emph{orthogonal} if each ordered symbol pair $(j,j')$ appears exactly $\lambda^2$ times in the superposition of $S$ on~$S'$.
\end{definition} 
\noindent
A set of pairwise orthogonal F-squares of type $\ff$ is a set of \emph{mutually orthogonal frequency squares} (MOFS) of type $\ff$. 
The central question is:
\[
\mbox{How large can a set of MOFS of type $\ff$ be?}
\]
The following result provides an upper bound.

\begin{theorem}[Hedayat, Raghavarao, Seiden (1975) \protect{\cite[Theorem 2.1]{hedayat-raghavarao-seiden}}]
\label{thm:upper-bound}
The number of MOFS of type $F(m\lambda;\lambda)$ is at most $\frac{(m\lambda -1)^2}{(m-1)}$. 
\end{theorem}
\noindent
A set of MOFS attaining the upper bound in \cref{thm:upper-bound} is \emph{complete}. 
The special case of a set of MOFS of type $F(m;1)$ is the well-known concept of a set of mutually orthogonal Latin squares (MOLS) of order $m$ (see \cite{abel-colbourn-dinitz} for background), and the existence of a complete set of $m-1$ MOLS of order $m$ is equivalent to the existence of a projective plane of order $m$ and an affine plane of order $m$ \cite[Theorem~3.20]{abel-colbourn-dinitz}. 
Several ideas from the study of sets of MOLS have been adapted to investigate the existence pattern for sets of MOFS. 

The following construction, which depends on symmetric factorial designs, provides a complete set of MOFS of type $\ff$ when $m$ is a prime power. 

\begin{theorem}[Hedayat, Raghavarao, Seiden (1975) \protect{\cite[Theorem 3.1]{hedayat-raghavarao-seiden}}]
\label{thm:complete-primepower}
Let $m$ be a prime power and $h$ a positive integer. Then there exists a complete set of $\frac{(m^h-1)^2}{m-1}$ MOFS of type $F(m^h;m^{h-1})$.
\end{theorem}
\noindent
Other combinatorial designs have been used to construct complete sets of MOFS with the same parameters as in \cref{thm:complete-primepower}, including linear permutation polynomials over GF$(m)$ \cite{mullen} and affine designs \cite{mavron}. 

A further construction of a complete set of MOFS of type $\ff$ depends on the existence of a Hadamard matrix of order $4n$ (which has long been conjectured for all positive integers $n$; see \cite{horadam-book} for background). 

\begin{theorem}[Federer (1977) \protect{\cite[Theorem 2.1]{federer}}]
\label{thm:complete-Hadamard}
Let $n$ be a positive integer and suppose a Hadamard matrix of order $4n$ exists. Then there exists a complete set of $(4n-1)^2$ MOFS of type $F(4n;2n)$. 
\end{theorem}

All known constructions of complete sets of F-squares of type $\ff$ having $m>2$ require $m$ to be a prime power, and the only known examples having $m=2$ are as described in \cref{thm:complete-Hadamard}. Jungnickel, Mavron and McDonough showed in 2001 how to unify all such constructions in terms of nets~\cite{jungnickel-mavron-mcdonough}.

Whereas \cref{thm:complete-Hadamard} shows the existence of a complete set of MOFS of type $F(2\lambda;\lambda)$ when $\lambda$ is even (subject to the existence of a Hadamard matrix of order $2\lambda$), a recent result established nonexistence when $\lambda > 1$ is odd by making a connection with resolvable designs.

\begin{theorem}[Britz, Cavenagh, Mammoliti, Wanless (2019+) \protect{\cite[Corollary~11]{britz}}]
\label{thm:no-complete}
There is no complete set of MOFS of type $F(2\lambda;\lambda)$ when $\lambda > 1$ is odd.
\end{theorem}
\noindent
In view of \cref{thm:no-complete}, it is natural to ask:
\begin{enumerate}[Q1.]
\item What is the maximum size of a set of MOFS of type $\ff$ when $\lambda > 1$ is odd?
\item
When is a set of MOFS of type $\ff$ maximal (that is, not extendible to a larger such set) but not complete?
\end{enumerate}
These questions are explored in the recent paper by Britz et al. \cite{britz} when $\lambda$ is odd. For example, it is shown computationally that for MOFS of type $F(6;3)$, the maximum size of a set is $17$ rather than the upper bound of $25$ given by \cref{thm:upper-bound}, and there are maximal sets of size $t$ for each $t$ satisfying $t \in \{1,17\}$ or $5 \le t \le 15$. 

In this paper, we introduce a representation of MOFS of type $\ff$ as a linear combination of $\{0,1\}$ arrays (\cref{sec:indicator-squares}). We use this representation to give a new elementary proof of the upper bound of \cref{thm:upper-bound}, together with a structural constraint on a complete set of MOFS (\cref{sec:new-proof}). 
We then use this representation to establish a maximality criterion for a set of MOFS of type $\ff$ when $m$ is even and $\lambda$ is odd (\cref{sec:maximality}), extending the analysis of \cite{britz} for the case when $m=2$ and $\lambda$ is odd.

\section{Indicator squares}
\label{sec:indicator-squares}

We begin by introducing a representation of an F-square of type $\ff$ as a linear combination of $\{0,1\}$ arrays. We shall use this representation in \cref{sec:new-proof} in the new proof of \cref{thm:upper-bound}, and in \cref{sec:maximality} to establish maximality criteria.
The \emph{indicator function} of a condition $X$ is the function
\[
I[X] = \begin{cases} 
 1 & \mbox{if $X$ holds,} \\ 
 0 & \mbox{otherwise}.
\end{cases}
\]

\begin{definition}
\label{defn:ind-square}
Let $S = (S_{ij})$ be an F-square of type $\ff$. 
For $a \in \{1,2,\dots,m\}$, the \emph{indicator square} $I_a(S)$ of $S$ with respect to $a$ is the $\{0,1\}$ array of size $m\lambda \times m\lambda$ whose $(i,j)$ entry is $I[S_{ij}=a]$.
\end{definition}
\noindent
Using \cref{defn:ind-square}, an F-square $S$ of type $\ff$ may be written as $\sum_{a=1}^m aI_a(S)$.

\begin{example}
Let 
\[
S =  \begin{bmatrix}
1& 	2& 	3&	1&	2&	3 \\
3&	1&	2&  	3& 	2&	1 \\
2& 	3&	1&  	2&	1& 	3 \\
1& 	1&	2&	3&	3& 	2 \\
3&	3& 	1& 	2&	1&	2 \\
2&	2&	3&  	1& 	3&	1 
\end{bmatrix}
\]
be an F-square of type $F(6;2)$. Then the indicator square of $S$ with respect to $1$, $2$, $3$ is
\[
I_1(S)  = \begin{bmatrix}
1	&0	&0	&1 	&0	&0 \\
0	&1 	&0 	&0	&0 	&1 \\
0	&0	&1      &0	&1 	&0 \\
1	&1    	&0      &0      &0  	&0 \\
0	&0	&1 	&0	&1 	&0 \\
0	&0	&0 	&1 	&0 	&1
\end{bmatrix}, \quad
I_2(S)  = \begin{bmatrix}
0	&1	&0	&0 	&1 	&0 \\
0	&0	&1     	&0	&1 	&0 \\
1	&0	&0	&1 	&0	&0 \\
0	&0      &1    	&0     	&0  	&1 \\
0	&0	&0 	&1 	&0 	&1 \\
1	&1      &0     	&0   	&0  	&0
\end{bmatrix}, 
\]
\[
I_3(S)  = \begin{bmatrix}
0	&0    	&1    	&0     	&0  	&1 \\
1	&0	&0	&1 	&0	&0 \\
0	&1 	&0 	&0	&0 	&1 \\
0	&0    	&0     	&1   	&1  	&0 \\
1	&1    	&0     	&0   	&0  	&0 \\
0	&0	&1 	&0	&1 	&0 
\end{bmatrix},
\]
respectively, and $S = I_1(S) + 2I_2(S) + 3I_3(S)$.
\end{example}

Let $A = (A_{ij})$ and $B = (B_{ij})$ be arrays of the same size. We write 
\[
A \circ B := \sum_{i,j} A_{ij} B_{ij}
\] 
for the sum of the entries of the elementwise product of $A$ and $B$. 
With this notation, $A \circ J$ is the sum of the entries of $A$ (where $J$ is the all-ones array whose size is given by context), and  $A \circ A$ is the sum of squares of the entries of $A$.
For an F-square $S$ of type $\ff$, we have
\begin{equation} 
\label{I(S)^2}
I_a(S) \circ I_a(S) = I_a(S) \circ J = m\lambda^2 \quad \mbox{for each $a \in \{1,2,\dots, m\}$},
\end{equation}
and F-squares $S, S'$ of type $\ff$ are orthogonal if and only if  
\begin{equation} 
\label{I(S)I(S')}
I_a(S) \circ I_b(S') = \lambda^2 \text{ for all } a,b \in \{1,2,\dots, m\}.
\end{equation}

\section{Proof of \cref{thm:upper-bound} using indicator squares}
\label{sec:new-proof}
The original proof of \cref{thm:upper-bound} by Hedayat, Raghavarao, and Seiden \cite{hedayat-raghavarao-seiden} uses an argument based on the matrix rank. 
Jungnickel, Mavron, and McDonough \cite[Theorem~3.5]{jungnickel-mavron-mcdonough} give an alternative counting proof, making use of an equivalence between sets of MOFS and nets, and characterizing the case when equality holds (namely, that the set of MOFS is complete).
In this section, we give a further elementary proof of \cref{thm:upper-bound} that provides a new structural constraint on the case when equality holds, using indicator squares. Our method is inspired by the proof of an upper bound for the number of rows of a difference matrix over a group given by van Greevenbroek and Jedwab \cite[Theorem~2.1]{vangreevenbroek-jedwab}, which is in turn adapted from a more general result due to Jungnickel \cite[Proposition~3.1]{jungnickel-diffmatrices}.

\begin{theorem}
\label{thm:T_structure}
Suppose $S_1,S_2,...,S_t$ is a set of MOFS of type $F(m\lambda; \lambda)$. Relabel the F-squares if necessary so that $(S_k)_{11} = 1$ for each $k \in \{1,2,\dots ,t\}$. Then $t \le \frac{(m\lambda-1)^2}{m-1}$, and equality holds if and only if 
\[
\sum_k\sum_{a>1}I_a(S_k) = \begin{bmatrix}
0 &\lambda(m\lambda-1) & \dots & \lambda(m\lambda-1) \\
\lambda(m\lambda-1) & \lambda(m\lambda-2)  & \dots  &\lambda(m\lambda-2) \\
\vdots  & \vdots  & \ddots & \vdots \\
\lambda(m\lambda-1)  &\lambda(m\lambda-2)  & \dots  &\lambda(m\lambda-2) 
\end{bmatrix}.
\]
\end{theorem}

\begin{proof}
Let $T=(T_{ij})$ be the $m\lambda \times m\lambda$ array given by 
 \[
 T = \sum_k\sum_{a>1}I_a(S_k).
 \]
We calculate the sum of the entries of $T$ as
\begin{align}
\sum_{i,j}T_{ij} &= T \circ J 		\nonumber \\
&= \sum_k \sum_{a>1} I_a(S_k) \circ J  	\nonumber \\
&= t(m-1)m\lambda^2			\label{eqn:T_sum}
\end{align}
from \eqref{I(S)^2}, and the sum of squares of the entries of $T$ as
\begin{align*}
\sum_{i,j} T_{ij}^2 = T \circ T &= \sum_{k,\ell} \sum_{a,b >1} I_a(S_k) \circ I_b(S_\ell) \\
&= \sum_{k \neq \ell} \sum_{a,b>1} I_a(S_k)\circ I_b(S_\ell) + \sum_k \sum_{a,b>1} I_a(S_k)\circ I_b(S_k) \\
&= t(t-1)(m-1)^2\lambda^2 + t(m-1)m\lambda^2
\end{align*}
by \eqref{I(S)I(S')} and \eqref{I(S)^2}, noting that $I_a(S_k)\circ I_b(S_k)  = 0$ for all $a \neq b$. Therefore 
\begin{align}\label{eqn: T^2}
\sum_{i,j}T_{ij}^2 = t(m-1)\lambda^2\big(t(m-1)+1\big).
\end{align}
Since $(S_k)_{11} = 1$ for each $k$, we have $T_{11} = 0$ and therefore
\begin{align}
\sum_{i>1} T_{i1} &= \sum_i T_{i1} \nonumber\\ 
&= \sum_k \sum_{a>1} \sum_i I_a(S_k)_{i1} \nonumber\\
&= t(m-1)\lambda 	\label{eqn:T_first_column_sum}
\end{align}
because $S_k$ is an F-square, and similarly
\begin{align}\label{eqn:T_first_row_sum}
\sum_{j>1} T_{1j} =  t(m-1)\lambda.
\end{align}

Now define an $m\lambda \times m\lambda$ array $U = (U_{ij})$ by 
\begin{equation*}
U_{ij} = 
\begin{cases} 
T_{11} +  \lambda(m\lambda - 2)	& \mbox{for $(i,j)=(1,1)$}, \\ 
T_{ij} - \lambda 		& \mbox{for $(i >1, j=1)$ and $(i = 1, j>1)$},\\
T_{ij}    			& \mbox{for $i,j > 1$}.
\end{cases}
\end{equation*}
We calculate the sum of the entries of $U$ as
\begin{align}
\sum_{i,j} U_{ij} 
 &= \sum_{i,j} T_{ij} +\lambda(m\lambda-2) - 2\lambda(m\lambda-1) \nonumber \\
 &= m\lambda^2 (t(m-1) - 1)					\label{eqn:sumU}
\end{align}
by substitution from \eqref{eqn:T_sum}, and the sum of squares of the entries of $U$ as
\begin{align*}
\sum_{i,j} U_{ij}^2 
 &= \Big(T_{11} +  \lambda(m\lambda - 2)\Big)^2 + \sum_{i>1}(T_{i1}-\lambda)^2 + \sum_{j>1}(T_{1j}-\lambda)^2 + \sum_{i,j > 1}T_{ij}^2 \\
 &= \sum_{i,j} T_{ij}^2 - 2\lambda\bigg( \sum_{i>1}T_{i1} + \sum_{j>1}T_{1j}\bigg) + \lambda^2(m\lambda -2)^2 + 2\lambda^2(m\lambda -1)\\
 &=t(m-1)\lambda^2\big(t(m-1)+1\big)- 2\lambda \cdot 2t(m-1)\lambda  + \\
 &\phantom{==} \lambda^2(m^2 \lambda^2 - 2 m\lambda +2)
\end{align*}
by substitution from \eqref{eqn: T^2}, \eqref{eqn:T_first_column_sum}, and \eqref{eqn:T_first_row_sum}. Therefore 
\begin{align*}
\sum_{i,j} U_{ij}^2 = \lambda^2 \Big (t^2(m - 1)^2-3t(m - 1) +m^2 \lambda^2 -2m\lambda+2\Big ).
\end{align*}
Substitute this and \eqref{eqn:sumU} into the Cauchy-Schwarz inequality
\begin{equation}
\label{eqn:CS}
\bigg(\sum_{i,j}U_{ij}\bigg)^2 \leq m^2\lambda^2\sum_{i,j}U_{ij}^2
\end{equation}
and simplify to obtain $t \leq \frac{(m\lambda -1)^2}{m -1}$.
Equality holds in \eqref{eqn:CS} if and only if the $U_{ij}$ are equal for all $i,j$, in which case from \eqref{eqn:sumU} and $t = \frac{(m\lambda -1)^2}{m -1}$ we obtain $U_{ij} = \lambda(m\lambda - 2)$ for all $i,j$. In that case, by definition of $U$ we have
\[
T = 
\begin{bmatrix}
0 &\lambda(m\lambda-1) & \dots & \lambda(m\lambda-1) \\
\lambda(m\lambda-1) & \lambda(m\lambda-2)  & \dots  &\lambda(m\lambda-2) \\
\vdots  & \vdots  & \ddots & \vdots \\
\lambda(m\lambda-1)  &\lambda(m\lambda-2)  & \dots  &\lambda(m\lambda-2)
\end{bmatrix}.
\]
\end{proof}

\section{A maximality criterion using indicator\\squares}
\label{sec:maximality}
In this section, we address question (Q2) of \cref{sec:intro}: when is a set of MOFS of type $\ff$ maximal (not extendible to a larger set) but not complete? As $m$ and $\lambda$ grow, it quickly becomes computationally infeasible to determine the maximality of a particular set of MOFS by direct comparison with all other F-squares. For example, the number of F-squares of type $F(6;3)$ is $297~200$ whereas the number of F-squares of type $F(8;4)$ is $116~963~796~250$~\cite{OEIS}. 
Nonetheless, Britz et al. \cite{britz} showed how to adapt parity arguments from the study of maximal sets of MOLS in order to obtain a theoretical criterion for the maximality of a set of MOFS of type $F(2\lambda;\lambda)$ when $\lambda$ is odd. They also derived necessary conditions on the MOFS parameters for the criterion to hold.
In this section, we extend the analysis of \cite{britz} to the case of MOFS of type $F(m\lambda;\lambda)$ for all even $m$ and odd $\lambda$, using indicator squares as introduced in \cref{sec:indicator-squares} and streamlining the arguments.

We shall derive a maximality criterion in \cref{thm:maximal_m}, depending on a sum of indicator squares having the regular block structure described in \cref{def:full_relation}.
Write $\mathbf{0}$ for the all-zeroes array whose size is given by context, and as before write $J$ for the all-ones array. For an array $A$, write $A \bmod 2$ for the elementwise reduction of $A$ modulo~$2$.
 
\begin{definition}
\label{def:full_relation}
Let $x,y$ be integers for which $0 \le x,y \le m\lambda$ and $x,y$ do not both belong to $\{0,m\lambda\}$. A set $\{S_1,S_2,\dots,S_t\}$ of F-squares of type $\ff$ satisfies a \emph{non-constant full relation} with respect to $x$ and $y$ if, for some permutation of rows and columns, the array
$\big(\sum_{k=1}^t I_1(S_k)\big) \bmod 2$ has block structure 
\begin{equation} 
\label{eqn:T_0mod2}
\begin{split}
\begin{tikzpicture}[scale=0.7]
\draw (0.2,0)--(0,0)--(0,4)--(0.2,4);
\draw (4.8,0)--(5,0)--(5,4)--(4.8,4);
\node at (-0.5,3) {$x$};
\node at (-1.4,1.3) {$m\lambda-x$};
\node at (1,4.5) {$y$};
\node at (3.5,4.5) {$m\lambda-y$};
\draw (2,0)--(2,4);
\draw (0.1,2.5)--(4.9,2.5);
\node at (1,3.3) {$\mathbf{0}$};
\node at (3.5,3.3) {$J$};
\node at (1,1.3) {$J$};
\node at (3.5,1.3) {$\mathbf{0}$};
\end{tikzpicture}
\end{split}
\mbox{ .}
\end{equation} 
\end{definition}

We present two preparatory results about a non-constant full relation in Lemmas~\ref{lem:xyequiv} and~\ref{lem:t_odd}. 

\begin{lemma}
\label{lem:xyequiv} 
Suppose that $\{S_1,S_2,\dots,S_t\}$ is a set of F-squares of type $\ff$ satisfying a non-constant full relation with respect to $x$ and~$y$. Then
\begin{enumerate}[(i)]
\item $x \equiv y \equiv t\lambda \pmod{2},$
\item $m\lambda \equiv 0 \pmod{2}.$
\end{enumerate}
\end{lemma}

\begin{proof}
Let $V = (V_{ij}) = \sum_k I_1(S_k)$. 
For each row $i$ of $V$, we have
$$
\sum_j V_{ij} = \sum_k \sum_j I_1(S_k)_{ij} = \sum_k \lambda
$$
because $S_k$ is an F-square of type $\ff$. Reduce modulo 2 to give
\begin{equation} \label{eqn:mod2rowsum}
\sum_j(V \bmod 2)_{ij} \equiv t\lambda \pmod{2} \quad \mbox{for each $i$}.
\end{equation}
Similarly, 
\begin{equation} \label{eqn:mod2columnsum}
\sum_i(V \bmod 2)_{ij} \equiv t\lambda \pmod{2} \quad \mbox{for each $j$}.
\end{equation}
By symmetry, we may assume that $0 < x < m\lambda$. With reference to \eqref{eqn:T_0mod2}, take $i = x+1$ in \eqref{eqn:mod2rowsum} to show that
\begin{equation}\label{eqn:yequiv}
y \equiv t\lambda \pmod{2},
\end{equation}
and take $i =1$ in \eqref{eqn:mod2rowsum} to show 
$$
m\lambda - y \equiv  t\lambda \pmod{2}.
$$
Combining with \eqref{eqn:yequiv} establishes $(ii)$.

Next take $j = 1$ in \eqref{eqn:mod2columnsum}. In the case $y = 0$, we obtain 
$$
x \equiv t\lambda \pmod{2},
$$ 
which with \eqref{eqn:yequiv} establishes $(i)$. Otherwise $y>0$, and then
$$
m\lambda - x \equiv t\lambda \pmod{2},
$$
which with $(ii)$ and \eqref{eqn:yequiv} establishes $(i)$.
\end{proof}

The condition in \cref{def:full_relation} that $x,y$ do not both belong to $\{0,m\lambda\}$ ensures that the array $\big(\sum_{k=1}^t I_1(S_k)\big) \bmod 2$ does not equal the constant array $\mathbf{0}$ or~$J$. 
Without this condition, both conclusions $(i)$ and $(ii)$ of \cref{lem:xyequiv} can fail, for example for the three F-squares of type $F(3;1)$ given by  
\[
S_1 = \begin{bmatrix}
1 &2 &3 \\
3 &1 &2 \\
2 &3 &1
\end{bmatrix}, \quad
S_2 = \begin{bmatrix}
2 &3 &1 \\
1 &2 &3 \\
3 &1 &2
\end{bmatrix}, \quad
S_3 = \begin{bmatrix}
3 &1 &2 \\
2 &3 &1 \\
1 &2 &3
\end{bmatrix},
\]
which satisfy 
\[
I_1(S_1) + I_1(S_2) + I_1(S_3) = \begin{bmatrix}
1 &1 &1 \\
1 &1 &1 \\
1 &1 &1
\end{bmatrix}
\]
and so fulfil all conditions of \cref{def:full_relation} with $(x,y)=(3,0)$ except that $x, y \in \{0,3\}$.

\begin{lemma}
\label{lem:t_odd}
Let $\lambda$ be odd, and suppose that $\{S_1,S_2,\dots,S_t\}$ is a set of MOFS of type $\ff$ satisfying a non-constant full relation. Then $t$ is odd. 
\end{lemma}

\begin{proof} 
Let the relation be with respect to $x$ and $y$, and let $V = \sum_k I_1(S_k)$. 
Using \eqref{I(S)I(S')}, the orthogonality of $S_1$ with each $S_k$ for $k > 1$ gives 
\begin{align*}
(t-1)\lambda^2 
  &= \sum_{k>1} I_1(S_k) \circ I_1(S_1)\\
  &= \sum_{k} I_1(S_k) \circ I_1(S_1) - I_1(S_1) \circ I_1(S_1)\\
  &= V \circ I_1(S_1) - m\lambda^2
\end{align*}
by \eqref{I(S)^2}. By \cref{lem:xyequiv}~$(ii)$ and the assumption that $\lambda$ is odd, reduction modulo 2 gives
\begin{equation} 
\label{eqn:s1prod_mod2}
t - 1 \equiv (V \bmod 2) \circ I_1(S_1) \pmod{2}. \\
\end{equation}
Since $S_1$ is an F-square of type $\ff$, the sum of the entries of the blocks of $I_1(S_1)$ corresponding to the blocks of $V \bmod 2$ shown in \eqref{eqn:T_0mod2} is
\begin{equation}
\label{eqn:F-square_blocks} 
\begin{split}
\begin{tikzpicture}[scale=0.7]
\draw (0.2,0)--(0,0)--(0,4)--(0.2,4);
\draw (6.8,0)--(7,0)--(7,4)--(6.8,4);
\node at (-0.5,3) {$x$};
\node at (-1.4,1.3) {$m\lambda-x$};
\node at (1,4.5) {$y$};
\node at (4.5,4.5) {$m\lambda-y$};
\draw (2,0)--(2,4);
\draw (0.1,2.5)--(6.9,2.5);
\node at (1,3.3) {$\alpha$};
\node at (4.5,3.3) {$x\lambda-\alpha$};
\node at (1,1.3) {$y\lambda-\alpha$};
\node at (4.5,1.3) {$m\lambda^2-(x+y)\lambda+\alpha$};
\end{tikzpicture}
\end{split}
\end{equation}
for some non-negative integer $\alpha$. 
Therefore from \eqref{eqn:s1prod_mod2} we obtain
\begin{align*}
t-1 
 &\equiv (x\lambda - \alpha) + (y\lambda -\alpha)  \\
 &\equiv 0 \pmod{2}
\end{align*}
by \cref{lem:xyequiv}~$(i)$, so $t$ is odd.
\end{proof}

We now use Lemmas~\ref{lem:xyequiv} and~\ref{lem:t_odd} to prove the desired maximality criterion.

\begin{theorem} 
\label{thm:maximal_m}
Let $\lambda$ be odd, and suppose that $\{S_1,S_2,\dots,S_t\}$ is a set of MOFS of type $\ff$ satisfying a non-constant full relation. Then $\{S_1,S_2,\dots,S_t\}$ is a maximal set of MOFS. 
\end{theorem}

\begin{proof}
Let the relation be with respect to $x$ and $y$, and let $V = \sum_k I_1(S_k)$. 
Suppose, for a contradiction, that $S$ is an F-square of type $\ff$ that is orthogonal to each~$S_k$. Then by \eqref{I(S)I(S')},
\[
t\lambda^2 = \sum_k  I_1(S_k) \circ I_1(S) = V \circ I_1(S).
\]
By \cref{lem:t_odd} and the assumption that $\lambda$ is odd, reduction modulo~$2$ gives
\begin{equation} \label{eqn:(T_0dotIR1)}
1 \equiv (V \bmod 2 ) \circ I_1(S) \pmod{2}.
\end{equation}
Since S is an F-square of type $\ff$, the sum of the entries of the blocks of $I_1(S)$ corresponding to the blocks of $V \bmod 2$ shown in \eqref{eqn:T_0mod2} is as shown in \eqref{eqn:F-square_blocks} for some non-negative integer~$\alpha$. Therefore from \eqref{eqn:(T_0dotIR1)},
\[
1 \equiv (x\lambda - \alpha)+ (y\lambda -\alpha) \equiv 0 \pmod{2}
\]
by \cref{lem:xyequiv}~$(i)$, which is a contradiction.
\end{proof}

The criterion of \cref{thm:maximal_m} for a set of MOFS of type $\ff$ to be maximal requires the set to satisfy a non-constant full relation, and we know from \cref{lem:xyequiv}~$(ii)$ that this requires $m$ to be even. 
We now derive a more restrictive necessary condition for this criterion to hold. 
\begin{proposition} 
\label{prop:mod8}
Let $\lambda$ be odd, and suppose that $\{S_1,S_2,\dots,S_t\}$ is a set MOFS of type $\ff$ satisfying a non-constant full relation with respect to $x$ and~$y$. Then $t \equiv m(x+y) - (m+1)\pmod{8}$. 
\end{proposition}

\begin{proof}
Let $V = \sum_k I_1(S_k)$, and let $x_r$ be the number of occurrences of $r$ in the array~$V$. We shall calculate expressions for the four quantities $\sum_r rx_r$ and $\sum_r r^2x_r$ and $\sum_{r\text{ odd}}x_r$ and $\sum_{r\text{ odd}}rx_r$, and substitute them into the congruence  
\begin{align} 
\sum_r (2r-r^2)x_r 
 & \equiv \sum_{r\text{ odd}} r(2-r)x_r \pmod{8}    \nonumber \\
 & \equiv \sum_{r\text{ odd}} (2r-1)x_r \pmod{8}  \label{eqn:sum_equiv_mod8}.
\end{align}

We have 
\begin{align*}
\sum_{r= 0}^t r x_r 
 &= V \circ J \\
 &= \sum_k I_1(S_k) \circ J \\
 &= tm\lambda^2
\end{align*}
by \eqref{I(S)^2}, and
\begin{align*}
\sum_{r=0}^t r^2x_r &= V \circ V \\
 &= \sum_{k,\ell} I_1(S_k) \circ I_1(S_\ell) \\
 &= \sum_{k \neq \ell} I_1(S_k) \circ I_1(S_\ell) + \sum_k I_1(S_k) \circ I_1(S_k) \\
 &=t(t-1)\lambda^2 + tm\lambda^2
\end{align*}
by \eqref{I(S)I(S')} and \eqref{I(S)^2}.

The expression $\sum_{r \text{ odd}} x_r$ is the sum of the entries of $V \bmod 2$, so with reference to the block structure of \eqref{eqn:T_0mod2} we have
\[
\sum_{r \text{ odd}} x_r = x(m\lambda - y) + y(m\lambda - x) .
\]

Since each $S_k$ is an F-square of type $\ff$, the sum of the entries of 
$V = \sum_{k=1}^t I_1(S_k)$ has the block structure
\begin{equation}
\label{eqn:T_0_blocks}
\begin{split}
\begin{tikzpicture}[scale=0.7]
\draw (0.2,0)--(0,0)--(0,4)--(0.2,4);
\draw (6.8,0)--(7,0)--(7,4)--(6.8,4);
\node at (-0.5,3) {$x$};
\node at (-1.4,1.3) {$m\lambda-x$};
\node at (1,4.5) {$y$};
\node at (4.5,4.5) {$m\lambda-y$};
\draw (2,0)--(2,4);
\draw (0.1,2.5)--(6.9,2.5);
\node at (1,3.3) {$\beta$};
\node at (4.5,3.3) {$tx\lambda-\beta$};
\node at (1,1.3) {$ty\lambda-\beta$};
\node at (4.5,1.3) {$tm\lambda^2-t(x+y)\lambda+\beta$};
\end{tikzpicture}
\end{split}
\end{equation}
for some non-negative integer $\beta$. 
The expression $\sum_{r \text{ odd}} r x_r$ is the sum of the odd entries of $V$, so
\begin{align*}\label{eqn:sumT_0_odd}
\sum_{r \text{ odd}} r x_r &= V \circ (V \bmod 2) \nonumber \\
&= (tx\lambda - \beta) + (ty\lambda - \beta)
\end{align*}
from \eqref{eqn:T_0_blocks} and \eqref{eqn:T_0mod2}.

Substitute the four calculated quantities into \eqref{eqn:sum_equiv_mod8} to give 
\begin{equation}
\label{eqn:mtlambda_mod8sum}
t\lambda^2(m+1-t) \equiv 2t(x+y)\lambda - 4\beta -m(x+y)\lambda + 2xy \pmod{8}.
\end{equation}
By comparison of \eqref{eqn:T_0_blocks} with \eqref{eqn:T_0mod2}, we see that $\beta$ is even. 
Also $\lambda$ is odd by assumption, and $m$ is even and $x, y, t$ are odd by Lemmas~\ref{lem:xyequiv} and~\ref{lem:t_odd}.
Note that for integers $a$ and $b$ with $a \equiv 0 \pmod{4}$ and b odd,
\begin{equation}
\label{ab_equiv_a}
ab\equiv a \equiv -a \pmod{8}. 
\end{equation}
Therefore \eqref{eqn:mtlambda_mod8sum} simplifies to 
\begin{equation*}
t(m+1)-1 \equiv m(x+y) + 2(x+y+xy) \pmod{8}.
\end{equation*}
Since $x$ and $y$ are odd  we have 
$2(x + y + xy) = 2(x+1)(y+1)-2 \equiv -2 \pmod{8}$, so that
\[
t(m+1) \equiv m(x+y) -1 \pmod{8}.
\]
Multiply by the odd integer $m+1$ and use \eqref{ab_equiv_a} to give 
\[
t \equiv m(x+y) - (m+1) \pmod{8}.
\]
\end{proof}

\begin{corollary}
\label{cor:mod4}
Let $\lambda$ be odd, and suppose that $\{S_1,S_2,\dots ,S_t\}$ is a set of MOFS of type $\ff$ satisfying a non-constant full relation. Then $t \equiv m-1 \pmod{4}$.
\end{corollary}

\begin{proof}
Let the relation be with respect to $x$ and $y$. By \cref{lem:xyequiv} we have that $x+y$ and $m$ are even. Reduce modulo~$4$ the conclusion of \cref{prop:mod8}.
\end{proof}

By taking the special case $m=2$ of \cref{thm:maximal_m}, \cref{prop:mod8}, \cref{cor:mod4} (and noting that in this case the condition in \cref{def:full_relation} that $x, y$ do not both belong to $\{0, m\lambda\}$ can be removed),
we recover Theorems 4, 6, 5 of \cite{britz}, respectively.
These theorems were obtained in \cite{britz} by regarding a set of MOFS of type $F(2\lambda;\lambda)$ as an orthogonal array and using counting arguments. The analysis presented here, using indicator squares, streamlines the arguments of \cite{britz} and allows us to deal with all even~$m>2$.

\section*{Acknowledgements}
The authors are grateful to Ian Wanless for helpful discussions at the \emph{14th International Conference on Finite Fields and their Applications} in Vancouver, BC in June 2019, and for kindly sharing extensive data on MOFS as well as a preliminary version of~\cite{britz}.

%

\end{document}